\documentclass[11pt]{amsart}
\usepackage[affil-it]{authblk}
\usepackage{tikz}
\usepackage{amsaddr}
\newcounter{braid}
\newcounter{strands}
\pgfkeyssetvalue{/tikz/braid height}{1cm}
\pgfkeyssetvalue{/tikz/braid width}{1cm}
\pgfkeyssetvalue{/tikz/braid start}{(0,0)}
\pgfkeyssetvalue{/tikz/braid colour}{black}
\pgfkeys{/tikz/strands/.code={\setcounter{strands}{#1}}}

\makeatletter
\def\cross{
\@ifnextchar^{\message{Got sup}\cross@sup}{\cross@sub}}
\def\cross@sup^#1_#2{\render@cross{#2}{#1}}
\def\cross@sub_#1{\@ifnextchar^{\cross@@sub{#1}}{\render@cross{#1}{1}}}
\def\cross@@sub#1^#2{\render@cross{#1}{#2}}

\def\render@cross#1#2{
\def\strand{#1}
\def\crossing{#2}
\pgfmathsetmacro{\cross@y}{-\value{braid}*\braid@h}
\pgfmathtruncatemacro{\nextstrand}{#1+1}
\foreach \thread in {1,...,\value{strands}}
{
\pgfmathsetmacro{\strand@x}{\thread * \braid@w}
\ifnum\thread=\strand
\pgfmathsetmacro{\over@x}{\strand * \braid@w + .5*(1 - \crossing) * \braid@w}
\pgfmathsetmacro{\under@x}{\strand * \braid@w + .5*(1 + \crossing) * \braid@w}
\draw[braid] \pgfkeysvalueof{/tikz/braid start} +(\under@x pt,\cross@y pt) to[out=-90,in=90] +(\over@x pt,\cross@y pt -\braid@h);
\draw[braid] \pgfkeysvalueof{/tikz/braid start} +(\over@x pt,\cross@y pt) to[out=-90,in=90] +(\under@x pt,\cross@y pt -\braid@h);
\else
\ifnum\thread=\nextstrand
\else
\draw[braid] \pgfkeysvalueof{/tikz/braid start} ++(\strand@x pt,\cross@y pt) -- ++(0,-\braid@h);
\fi
\fi
}
\stepcounter{braid}
}

\tikzset{braid/.style={double=\pgfkeysvalueof{/tikz/braid colour},double distance=1pt,line width=2pt,white}}

\newcommand{\braid}[2][]{%
\begingroup
\pgfkeys{/tikz/strands=2}
\tikzset{#1}
\pgfkeysgetvalue{/tikz/braid width}{\braid@w}
\pgfkeysgetvalue{/tikz/braid height}{\braid@h}
\setcounter{braid}{0}
\let\sigma=\cross
#2
\endgroup
}
\makeatother

\usepackage[utf8]{inputenc}
\usepackage[letterpaper,margin=1in,footskip=0.5in]{geometry}
\usepackage{graphicx}
\usepackage{booktabs}
\usepackage{array}
\usepackage{paralist}
\usepackage{verbatim}
\usepackage{subfig}
\usepackage{amsfonts}
\usepackage{amsmath}
\usepackage{amsthm}
\usepackage{amssymb}
\usepackage{latexsym}

\newtheorem{remark}{Remark}

\newtheorem{theorem}{Theorem}
\newtheorem*{theorem*}{Theorem}
\newtheorem{lemma}{Lemma}

\title{Resolution depth of positive braids}
\author{Elliot Kaplan$^1$, David Krcatovich$^2$ and Patricia O'Brien$^3$}
\address{$^1$Ohio University, Athens, Ohio 45701}
\address{$^2$Michigan State University, East Lansing, Michigan 48824}
\address{$^3$University of Texas, Austin, Texas 78712}
\email{ek432210@ohio.edu}
\email{krcatov6@msu.edu}
\email{pjobrien@utexas.edu}
\begin{document}
\maketitle
\begin{abstract}
The depth of a link measures the minimum height of a resolving tree for the link whose leaves are all unlinks. We show that the depth of the closure of a strictly positive braid word is the length of the word minus the number of distinct letters. 
\end{abstract}
\section{Introduction}
\label{sec:introduction}

It was shown by Conway \cite{Conway} that the Alexander polynomial of an oriented link can be computed through a skein relation. If $L_+$, $L_-$ and $L_0$ are three link diagrams which are identical except in the neighborhood of a point, where they differ as in Figure \ref {skeinknot}, then the relationship 
\begin{equation}
\Delta_{L_+}(t) - \Delta_{L_-}(t)= (t^{-1/2} - t^{1/2})\Delta_{L_0}(t),
\label{skein}
\end{equation}
 along with the fact that $\Delta_{\text{unknot}}(t) = 1$, determines the Alexander polynomial for any oriented link.
\begin{figure}[h!]
\centering
\parbox{5.5cm}{
\centering
\includegraphics[width=5cm]{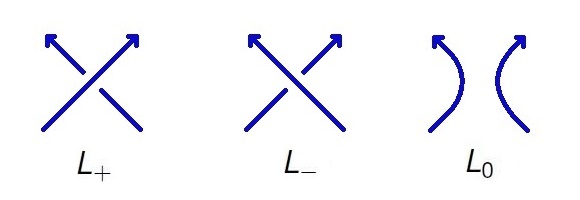}
}
\qquad \qquad \qquad
\begin{minipage}{5.5cm}
\centering
\includegraphics[width=5cm]{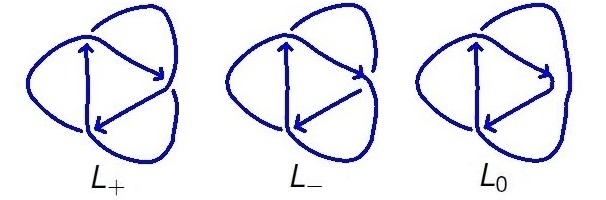}
\end{minipage}
\caption{Three types of crossings and three link projections with these crossings}
\label{skeinknot}
\end{figure}
Similar skein relations can be used to compute the Jones and HOMFLY-PT polynomials of a link.

Iterating this skein relation to compute a link polynomial leads one to construct a \textit{resolving tree} \cite{Adams} (also called a \textit{skein tree} \cite{Thompson}, or \textit{computation tree} \cite{FranksWilliams}). A resolving tree for a link $L$ is a binary tree with $L$ as the root. A diagram $D$ is chosen, and a crossing is selected, so that $D$ serves as $L_{\pm}$; the left child is the diagram with the crossing changed, $L_{\mp}$, and the right is the diagram with the crossing resolved, $L_0$. This is repeated at each node (possibly after an isotopy of the diagram) which is not an unlink. With a deliberate choice of crossings, this process terminates, giving a finite tree. The depth of a leaf is the length of the (shortest) path from it to the root. The depth of a resolving tree is the maximum depth among all leaves. An example of a resolving tree with depth 6 is given in Figure \ref{72tree}. The depth of a link is the minimal depth among all resolving trees for the link. As such, it gives a measure of the complexity of computing link polynomials via this skein relation. \\
\begin{figure}[h!]
\centering
\includegraphics[width=15cm]{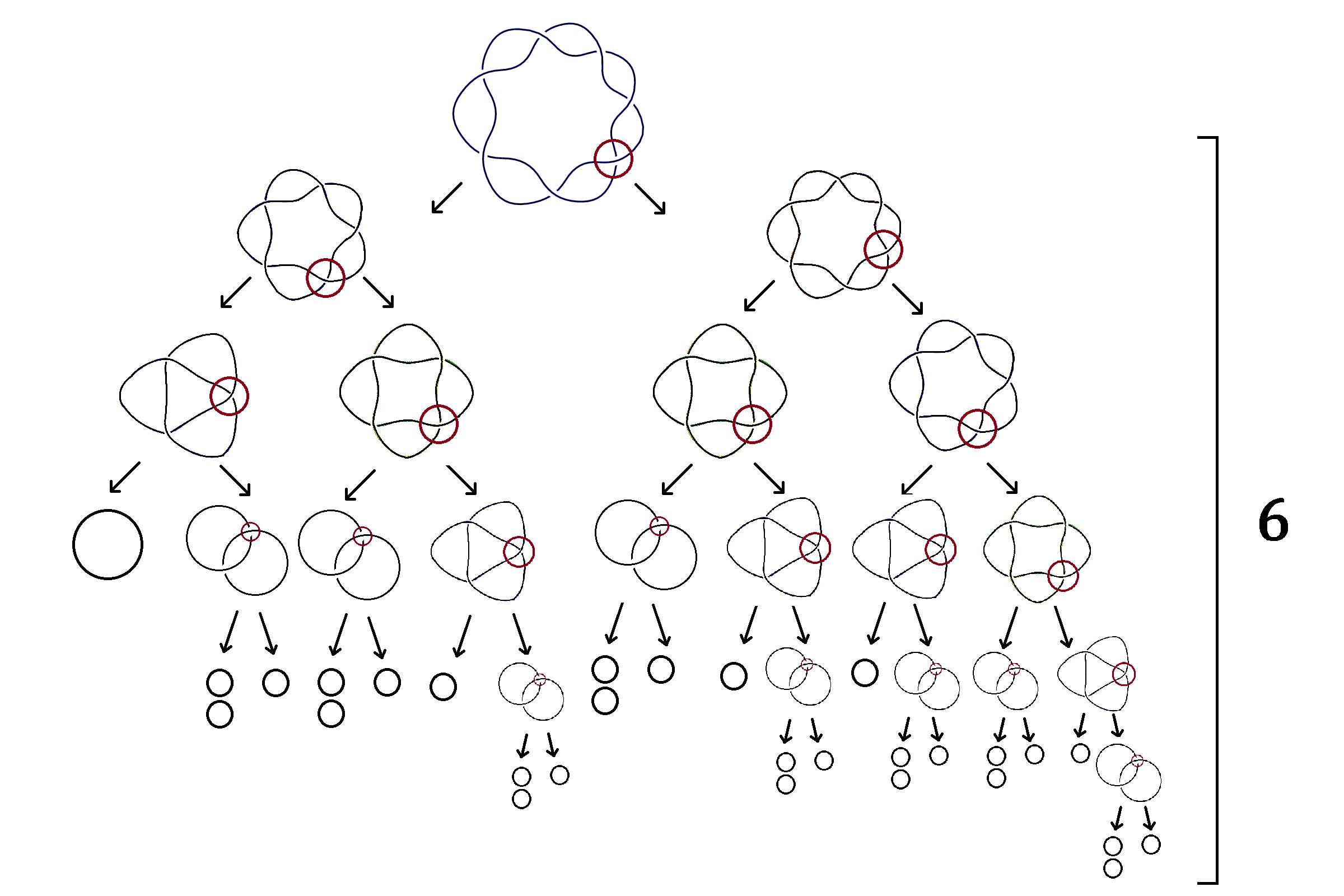}
\caption{A resolving tree for the $(7,2)$-torus knot}
\label{72tree}
\end{figure}

The leftmost branch of a resolving tree consists entirely of crossing changes, and therefore the unlinking number is a lower bound on the depth of a link. Note also that the skein relation \eqref{skein} implies that the breadth of the Alexander polynomial is also a lower bound on the depth. Finding the depth of a particular link may require finding a resolving tree whose depth agrees with one of these lower bounds, which is not generally easy.\\

Links which have depth 1 were classified in \cite{ScharlemannThompson}; those of depth 2 were classified in \cite{Thompson}. In this paper, we determine the depth of any link which is the closure of a positive (or negative) braid, by finding an explicit way to construct a resolving tree whose depth agrees with the lower bound coming from the Alexander polynomial. 
\begin{theorem}
Suppose that an oriented link $L$ is the closure of a strictly positive (negative) braid in $B_n$ which is represented by a word of length $\ell$. Then the depth of $L$ is $\ell-n+1$.
\label{main}
\end{theorem}

Section \ref{sec:braids} gives the necessary background on braid closures and describes the effect which crossing changes and resolutions have on the corresponding braid word. Section \ref{sec:upper} shows how to construct a particular resolving tree, giving an upper bound on the depth in Lemma \ref{upperbound}. Section \ref{sec:lower}  shows (Lemmas \ref{state1} and \ref{statet}) that this agrees with a lower bound coming from the Alexander polynomial, which proves Theorem \ref{main}.

\textbf{Acknowledgments.} This work was done with support from the Summer Undergraduate Research Institute in Experimental Mathematics (SURIEM) at Michigan State University, in which the first and third authors were participants and the second author a graduate assistant. We are sincerely indebted to Teena Gerhardt, the project mentor, for her guidance throughout the summer. This was funded in part by the NSA and NSF grant DMS-1062817.
\section{Braids}
\label{sec:braids}
\subsection{Manipulating Braids}

Let $B_n$ denote the braid group on $n$ strands, with generators $\sigma_1,\ldots, \sigma_{n-1}$. It will be convenient for an inductive argument in Section \ref{sec:upper} to let $B_n[p]$ denote the subgroup of $B_n$ generated by $\sigma_1,\ldots, \sigma_{p-1}$ (i.e., $n$-stranded braids where at most the first $p$ strands are braided, and the remaining are split unknot components). For future reference we recall the relations of $B_n$:

\begin{itemize}
\item[] \textbf{Rule 1:} $ \sigma_i\ \sigma_i^{-1} = 1 = \sigma_i^{-1}\ \sigma_i$
\item[] \textbf{Rule 2:} $ \sigma_i\ \sigma_{i+1}\ \sigma_i =\ \sigma_{i+1}\ \sigma_i\ \sigma_{i+1}$
\item[] \textbf{Rule 3:} If $\mid i - j \mid > 1$, then $ \sigma_i\ \sigma_j\  =\ \sigma_j\ \sigma_i\  $.
\end{itemize}
Further, if $\omega$ is a word in $B_n$, let $\widehat{\omega}$ denote its closure, a link in $S^3$. The following ways \cite{Markov} of changing $\omega$ do not change the isotopy class of $\widehat{\omega}$:
\begin{itemize}
\item[] \textbf{Rule 4 ((De-)stabilization):} If $\omega\in B_n$ and $\omega'=\omega \sigma_n \in B_{n+1}$,  then $\widehat{\omega}$ is equivalent to $\widehat{\omega'}$ as a link. More generally, if $\omega \in B_n[p]$ and $\omega'=\omega \sigma_p \in B_{n+1}[p+1]$, then $\widehat{\omega}$ is equivalent to  $\widehat{\omega'}$.
\item[] \textbf{Rule 5 (Conjugation):} For any $\omega, \eta \in B_n$, $\widehat{\eta \omega \eta^{-1}}$ is equivalent to $\widehat{\omega}$ as a link.
\end{itemize}
\begin{remark}
\label{braidskein}
An occurence of $\sigma_i$ in a braid word $\omega$ represents an $L_+$ crossing in $\widehat{\omega}$, and similarly $\sigma_i^{-1}$ represents an $L_-$ crossing. A crossing change results in the switching of the exponent from $1$ to $-1$, or vice versa. A crossing resolution results in the elimination of this occurence of $\sigma_i$ from $\omega$.
\end{remark}

This remark is justified by the following diagram, where $\omega_1$ and $\omega_2$ are arbitrary braid words in $B_n$ (for some $n\geq 2$).


\begin{figure}[h]
\begin{center}
\begin{tikzpicture}
	\braid[strands=1,braid start={(0,0)}]
	{\sigma_1^{-1}}
	\node at (1.5, -1.5) {$\omega_1\ \sigma_i\ \omega_2$};
	\node at (4.4, -1.5) {$\omega_1\ \sigma_i^{-1}\ \omega_2$};
	\node at (7.1, -1.5) {$\omega_1 \omega_2$};
	\node at (1.5, .5) {$L_+$};
	\node at (4.3, .5) {$L_-$};
	\node at (7, .5) {$L_0$};
	\braid[strands=2,braid start={(2.8,0)}]
	{\sigma_1}
	\tikzset{translate/.style={shift={(#1)}}}
	\draw[very thick]
	([translate=-40:1cm]5.75,-.5) arc (-40:40:1cm)
	([translate=-40+180:1cm]8.25,-.5) arc (-40+180:40+180:1cm);
	\node at (1,0) {$\wedge$};
	\node at (2,0) {$\wedge$};
	\node at (3.8,0) {$\wedge$};
	\node at (4.8,0) {$\wedge$};
	\node at (6.75,-.5) {$\wedge$};
	\node at (7.25,-.5) {$\wedge$};
\end{tikzpicture}
\end{center}
\end{figure}
By construction, all strands of a braid are oriented in the same direction, so this picture is sufficiently general.

\begin{remark}
Let $\omega$ be a braid word in which, for some $i$, $\sigma_i$ appears two consectutive times (i.e., we have $\sigma_i^2$). If we target the crossing in $\widehat{\omega}$ which is represented by one of these $\sigma_i$'s and change the diagram from $L_+$ to $L_-$, the braid word representing our new link is the result of removing $\sigma_i^2$ from $\omega$. If we instead resolve the crossing, this amounts to replacing $\sigma_i^2$ with $\sigma_i$.
\label{rmk:square}
\end{remark}
\subsection{Strictly Positive Braids}
A \textit{positive braid word} is one with no negative exponents, for example, $\sigma_ 3 \sigma_1 \sigma_3 \sigma_2$. A \textit{positive braid} is one which can be represented by a positive braid word. For example, $\sigma_ 3^{-1} \sigma_1 \sigma_3 \sigma_2$ represents a positive braid (because it is equivalent to $\sigma_1 \sigma_2$), although it is not a positive braid word. A \textit{strictly positive braid} is a braid which can be represented by a positive braid word in $B_n$ such that $\sigma_i$ appears at least once for each $1\leq i\leq n-1$. Note that strict positivity is a characteristic of the braid closure -- the closure of a positive braid is the closure of a strictly positive braid if and only if it is a non-split link. An \textit{essentially strictly positive braid} is one which can be represented by a positive braid word $\omega \in B_n[p]$ for some $p$, such that $\sigma_i$ appears at least once for each $1\leq i\leq p-1$. In this case, we will say $\omega$ is \textit{strictly positive in} $B_n[p]$. Note that essentially strict positivity is also a characteristic of the braid closure -- the closure of a positive braid is the closure of an essentially strictly positive braid if and only if the first $p$ strands form a non-split link and each of the remaining $n-p$ strands is an unknot, for some $p$.

The following Lemma, paired with Remark \ref{rmk:square}, is the key to finding a resolving tree whose complexity decreases at each step. It is a slight variation of Lemma 2.1 of \cite{FranksWilliams} (cf. \cite[p. 34]{Rudolph}).
\begin{lemma}
\label{doublesigma}

Suppose $\omega \in B_n$ is an essentially strictly positive braid whose closure $\widehat{\omega}$ is not an unlink. Then $\widehat{\omega}$ is the closure of an essentially strictly positive braid word in $B_m[p]$ (for some $p\leq m\leq n$) which contains $\sigma_i^2$ for some $i$. Further, such a word can be chosen so that either $i=p$, or there are more than two occurences of $\sigma_i$.

\end{lemma}
\begin{proof}[Proof of the first claim:]

It is clear that we may assume $\omega$ is strictly positive, for if a word in $B_n[p]$ contains a square, so does the same word in $B_p$ (and this word's closure is either an unlink in both cases or in neither case).

So, let $\omega_0$ be a strictly positive braid word which represents $\omega \in B_n$. Without loss of generality we assume that we have at least two $\sigma_{n-1}$'s in $\omega_0$, for if this is not the case, we can destabilize to get an strictly positive braid word in $B_{n-1}$, and then destabilize repeatedly, as many times as possible. A sequence of such destabilizations gives either the empty braid word in $B_1$, whose closure is the unknot, or a positive braid word in $B_i$ in which $\sigma_{i-1}$ appears at least twice.

We begin by targeting two $\sigma_{n-1}$'s which have no other $\sigma_{n-1}$'s between them in $\omega_0$. Using Rule 3, we will move our $\sigma_{n-1}$'s as close together as possible. After this is done, we are left with three possible configurations:
\begin{itemize}
\item[] \textbf{Case 1:} $\cdots\ \sigma_{n-1}\ \sigma_{n-1}\ \cdots$\\
In this case, a series of Rule 3 moves was sufficient to bring our $\sigma_{n-1}$'s together. Our first claim is satisfied.
\item[] \textbf{Case 2:} $\cdots\ \sigma_{n-1}\ \sigma_{n-2}\ \sigma_{n-1}\ \cdots$\\
Here, we apply Rule 2 to change $\omega$ to $\cdots\ \sigma_{n-2}\ \sigma_{n-1}\ \sigma_{n-2}\ \cdots$. If the resulting $\sigma_{n-1}$ is the only $\sigma_{n-1}$ left, we remove it through destabilization, resulting in $\cdots\ \sigma_{n-2}^2\ \cdots$. Otherwise, we repeat the process with two of the remaining $\sigma_{n-1}$'s (again with no other $\sigma_{n-1}$'s between them). Note that the number of $\sigma_{n-1}$'s in $\omega$ has decreased by one, so after some number of repetitions, we will no longer have to consider this case.
\item[] \textbf{Case 3:} $\cdots\ \sigma_{n-1}\ \sigma_{n-2}\ \cdots\ \cdots\ \sigma_{n-2}\ \sigma_{n-1}\ \cdots$\\
Here our two $\sigma_{n-1}$'s are seperated by a word in $B_{n-1}$, which we will refer to as $\omega_1$,
and $\omega_1$ has at least two $\sigma_{n-2}$'s. We choose two $\sigma_{n-2}$'s with no other $\sigma_{n-2}$'s between them, and attempt to bring these together using Rule 3.\\
\end{itemize}
We can iterate this process, either finding a square (Case 1 or 2) or representing $\widehat{\omega}$ by a word $\omega_i \in B_{n-i}$ for successively larger $i$. Eventually, if we have not produced a square at a previous step, we end up with $\omega_{n-2} \in B_2$ containing at least two $\sigma_1$'s, so we have $\sigma_1^2$.\\

\textit{Proof of the second claim:}

Suppose we have represented $\widehat{\omega}$ by a strictly positive word $\omega' \in B_{n}$ which contains $\sigma_i^2$ for some $i<n-1$, and no other occurences of $\sigma_i$. By conjugating, we may assume $\omega' = \sigma_i^2 \eta$, where $\eta$ is a braid word in $B_n$ not containing $\sigma_i$. Then, using Rule 3, we may write $\omega' = \sigma_i^2 \eta_- \eta_+$, where $\eta_-$ is a word consisting of letters $\sigma_j$ for $j<i$, and $\eta_+$ is a word consisting of letters $\sigma_j$ for $j>i$. Since $\omega'$ is strictly positive, $\eta_+$ is not the empty word, and in fact contains all letters $\sigma_{i+1},\cdots, \sigma_{m-1}$. Applying the argument in the proof of the first claim to $\eta_+$, we can either destabilize all of $\eta_+$ or we can find $\sigma_j^2$ for some $j>i$.

\end{proof}
\section{An Upper Bound on Depth}
\label{sec:upper}
Let $\widehat{\omega_0}$ be the closure of a strictly positive braid word in $B_n[p]$. The aim of this section is to show that we can construct a resolving tree for which the ``complexity" decreases each step down from the root. Recursively, if $\widehat{\omega_{i}^\alpha}$ is the closure of a strictly positive braid word in $B_n[p_i^{\alpha}]$, where $\alpha$ is a binary sequence of length $i$, then after choosing a crossing, define $\widehat{\omega_{i+1}^{\alpha,0}}$ to be the closure of a braid word resulting from an $L_+ \to L_-$ crossing change, and $\widehat{\omega_{i+1}^{\alpha,1}}$ the closure resulting from an $L_{+} \rightarrow L_0$ resolution. Note that by Lemma \ref{doublesigma} and Remark \ref{rmk:square}, a crossing can always be chosen so that both $\widehat{\omega_{i+1}^{\alpha,1}}$ and $\widehat{\omega_{i+1}^{\alpha,0}}$ are also closures of essentially strictly positive braid words. Let $\ell_i^{\alpha}$ denote the length of $\omega_i^{\alpha}$, and let $\chi_i^{\alpha} = \ell_i^{\alpha} - p_i^{\alpha}+1$. We will view $\chi$ as the complexity of the word, since it is merely the length minus the number of distinct letters. Note that if $\widehat{\omega_i^\alpha}$ is the closure of a strictly positive braid in $B_n$, then $p_i^\alpha= n$.
\begin{lemma}
\label{length}
Let $\widehat{\omega_i^\alpha}$ be the closure of a strictly positive braid word in $B_n[p_i^{\alpha}]$ which contains $\sigma_j^2$ for some $j$, such that either there are more than two $\sigma_j$'s or $j = p_i^{\alpha}-1$. If we select a crossing which corresponds to one of these $\sigma_j$'s, then \[\max \left\{ \chi_{i+1}^{\alpha,0}, \chi_{i+1}^{\alpha,1} \right\} \leq \chi_i^\alpha -1.\] 
\end{lemma}

Paired with Lemma \ref{doublesigma}, this says we can construct a resolving tree for the closure of an essentially strictly positive braid for which the complexity decreases by at least one at each step.
\begin{proof}

For the left child, we change one of these $\sigma_j$'s to $\sigma_j^{-1}$
We have three cases for this braid closure and its left child:
\begin{enumerate}
\item There are exactly two $\sigma_j$'s remaining.\\
By our premise, $j= p_i^{\alpha}-1$. We change one of the remaining $\sigma_j$'s to $\sigma_j^{-1}$, which then cancels with the last $\sigma_j$. The result is the closure of a strictly positive braid in $B_n[p_i^{\alpha}-1]$. In this case, $\ell_{i+1}^{\alpha,0} = \ell_i^{\alpha} - 2$ and $p_{i+1}^{\alpha,0} = p_i^{\alpha}-1$.
\item  There are exactly three $\sigma_j$'s remaining and $j= p_i^{\alpha}-1$.\\
We change one $\sigma_j$ to $\sigma_j^{-1}$, which then cancels with another $\sigma_j$, and the third $\sigma_j$ can be eliminated through destabilization, resulting in the closure of a strictly positive braid in $B_{n-1}[p_i^{\alpha}-1]$. So, $\ell_{i+1}^{\alpha,0} = \ell_i^{\alpha} - 3$ and $p_{i+1}^{\alpha,0} = p_i^{\alpha}-1$.
\item There are more than three $\sigma_j$'s remaining (or at least three if $j\neq p_i^{\alpha}-1$).\\
We change one $\sigma_j$ to $\sigma_j^{-1}$, which then cancels with another $\sigma_j$, but there are still at least two $\sigma_j$'s remaining. If $j\neq p_i^{\alpha}-1$ and there is only one $\sigma_j$ remaining, this still can not be eliminated through destabilization, so $\ell_{i+1}^{\alpha,0} = \ell_i^{\alpha} - 2$ and $p_{i+1}^{\alpha,0} = p_i^{\alpha}$.
\end{enumerate}
In the second and third cases, $\chi$ decreases by 2. In the first case, $\chi$ decreases by 1.\\
We have two cases for this braid closure and its right child:
\begin{enumerate}
\item There are exactly two $\sigma_j$'s remaining.\\
By assumption, $j= p_i^{\alpha}-1$. We remove one $\sigma_j$ and the other $\sigma_j$ can be eliminated through destabilization, so $\ell_{i+1}^{\alpha,1} = \ell_i^{\alpha} - 2$ and $p_{i+1}^{\alpha,1} = p_i^{\alpha}-1$.
\item There are more than two $\sigma_j$'s remaining.
We remove one $\sigma_j$ and there are still at least two $\sigma_j$'s remaining, so $\ell_{i+1}^{\alpha,1} = \ell_i^{\alpha} - 1$ and $p_{i+1}^{\alpha,1} = p_i^{\alpha}$.
\end{enumerate}
In both cases, $\chi$ decreases by 1.
\end{proof}

\begin{lemma}
Every link arising from the closure of a strictly positive braid word $\omega\in B_n[p]$ has a resolving tree of depth less than or equal to $\ell - p+1$, where $\ell$ is the length of $\omega$. Thus, the depth of the closure $\widehat{\omega}$ has an upper bound of $\ell - p+1$.
\label{upperbound}
\end{lemma}
\begin{proof}
After using the first claim of Lemma \ref{doublesigma} to find squares, Lemma \ref{length} allows us to construct a tree for $\widehat{\omega}$ of depth less than or equal to $\chi=\ell-p+1$, such that for each leaf we have $\chi_i^{\alpha}=0$. It remains to see that this is a complete resolving tree -- in other words, that these leaves are in fact all unlinks. But $\chi_i^{\alpha}=0$ implies that each letter $\sigma_1,\ldots,\sigma_{p_{i}^{\alpha}-1}$ is used exactly once, so successive destabilizations result in a braid closure with no crossings.

\end{proof}
\section{A Lower Bound on Depth}
\label{sec:lower}
The skein relation \eqref{skein} implies that if \(L'\) is a link which appears as a child of \(L\) in a resolving tree, then
\begin{equation}
Br(\Delta_L(t))\leq Br(\Delta_{L'}(t))+1,
\end{equation}
where \(Br\) is the breadth of a polynomial (the difference between the greatest and least exponents of its nonzero terms). Since an unlink has an Alexander polynomial of breadth zero, it follows that
\begin{equation}
Br(\Delta_L(t))\leq \text{depth}(L).
\label{breadthineq}
\end{equation}
Therefore, to prove our main result, it is only necessary to show that if $L$ is the closure of a strictly positive braid, represented by a braid word $\omega$ of length $\ell$ in $B_n$, then \[Br(\Delta_L(t))\geq \ell-n+1.\] This is a well-known fact, because closures of positive braids are fibered \cite{Stallings} , but we present a proof here for the sake of self-containment. The proof relies on Kauffman's state sum model for the Alexander polynomial \cite{Kauffman}.

We will prove this by showing that (after a choice of starred regions) there is a unique Kauffman state with weight \(\pm1\), and a unique Kauffman state with weight \(\pm t^{\ell-n+1}\). In order to do so, we introduce some notation, with Figure \ref{refbraid} as an illustration. Let $n_i$ be the number of occurences of $\sigma_i$ in $\omega$. Let \(c_i^j\) denote the crossing which corresponds to the \(j\)th occurence of $\sigma_i$. Let $R_i^j$ denote the region which has $c_i^j$ at the top and $c_i^{j-1}$ at the bottom, if $j>1$, and $R_i^1$ the region which connects $c_i^{n_i}$ to $c_i^1$. Finally, $R_0$ is the exterior region, and $R_n$ is the interior region. Star the adjacent regions $R_0$ and $R_1^1$.\\

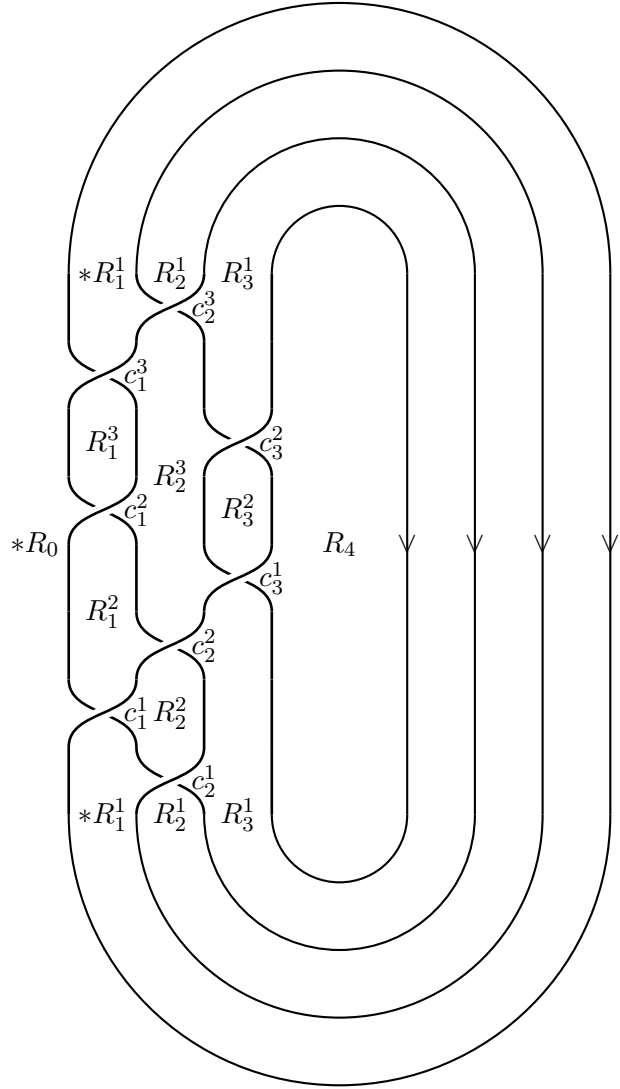
\begin{figure}[h!]
\centering
\begin{tikzpicture}[scale = .9]
\braid[strands=4,braid start={(0,0)}]
{\sigma_2^{-1} \sigma_1^{-1} \sigma_3^{-1} \sigma_1^{-1} \sigma_3^{-1} \sigma_2^{-1} \sigma_1^{-1} \sigma_2^{-1}}
\draw[thick] (6,0)--(6,-8);
\draw[thick] (7,0)--(7,-8);
\draw[thick] (8,0)--(8,-8);
\draw[thick] (9,0)--(9,-8);
\draw[thick] (6,0) arc (0:180:1);
\draw[thick] (7,0) arc (0:180:2);
\draw[thick] (8,0) arc (0:180:3);
\draw[thick] (9,0) arc (0:180:4);
\draw[thick] (6,-8) arc (180:0:-1);
\draw[thick] (7,-8) arc (180:0:-2);
\draw[thick] (8,-8) arc (180:0:-3);
\draw[thick] (9,-8) arc (180:0:-4);
\node at (6,-4) {$\vee$};
\node at (7,-4) {$\vee$};
\node at (8,-4) {$\vee$};
\node at (9,-4) {$\vee$};
\node at (2,-6.5) {$c_1^1$};
\node at (2,-3.5) {$c_1^2$};
\node at (2,-1.5) {$c_1^3$};
\node at (3,-7.5) {$c_2^1$};
\node at (3,-5.5) {$c_2^2$};
\node at (3,-0.5) {$c_2^3$};
\node at (4,-4.5) {$c_3^1$};
\node at (4,-2.5) {$c_3^2$};
\node at (0.5,-4) {$\ast R_0$};
\node at (1.5,-8) {$\ast R_1^1$};
\node at (1.5,0) {$\ast R_1^1$};
\node at (1.5,-5) {$R_1^2$};
\node at (1.5,-2.5) {$R_1^3$};
\node at (2.5,-8) {$R_2^1$};
\node at (2.5,0) {$R_2^1$};
\node at (2.5,-6.5) {$R_2^2$};
\node at (2.5,-3) {$R_2^3$};
\node at (3.5,-8) {$R_3^1$};
\node at (3.5,-3.5) {$R_3^2$};
\node at (3.5,0) {$R_3^1$};
\node at (5,-4) {$R_4$};
\end{tikzpicture}
\caption{The closure of the braid $\sigma_2 \sigma_1 \sigma_2 \sigma_3 \sigma_1 \sigma_3 \sigma_1 \sigma_2$}
\label{refbraid}
\end{figure}

\begin{lemma}
There is a unique state with weight $\pm 1$.
\label{state1}
\end{lemma}
\begin{proof}
Figure \ref{genrij} shows a generic region $R_i^j$. Note that it has a single $-1$ at its bottom crossing, a single $-t$ at its top crossing, $1$ at each of its left crossings, and $t$ at each of its right crossings. Start with the crossing $c_1^{n_1}$. Since the two starred regions are adjacent to this crossing, its only remaining labels are $1$, in region $R_2^j$ for some $j$, and $-t$. Therefore, to get a weight of $\pm 1$, we must assign the region $R_2^j$ to this crossing. Next, consider the crossing $c_2^{j-1}$. It has a label of $-1$ in region $R_2^j$, but we have already used this region; therefore we are forced to use the label $1$ in region $R_3^k$, for some $k$. This process repeats, forcing us to use a label 1 in region $R_n$ at some crossing $c_{n-1}^l$. To this point, we have positioned the black circles in Figure \ref{kauffmanpm1}. \\

\begin{figure}[h!]
\centering
\parbox{5.5cm}{
\centering
\begin{tikzpicture}
\braid[strands=1,braid start={(1,0)}]
{\sigma_1^{-1}}
\braid[strands=1,braid start={(0,-1)}]
{\sigma_1^{-1}}
\braid[strands=1,braid start={(2,-2)}]
{\sigma_1^{-1}}
\braid[strands=1,braid start={(0,-3)}]
{\sigma_1^{-1}}
\braid[strands=1,braid start={(2,-4)}]
{\sigma_1^{-1}}
\braid[strands=1,braid start={(1,-5)}]
{\sigma_1^{-1}}
\draw[thick] (2,-2)--(2,-3);
\draw[thick] (2,-4)--(2,-5);
\draw[thick] (3,-1)--(3,-2);
\draw[thick] (3,-3)--(3,-4);
\node at (2,-2.5) {$\wedge$};
\node at (3,-3.5) {$\wedge$};
\node at (2.5,-6) {$c_i^{j-1}$};
\node at (2.5,-3) {$R_i^j$};
\node at (2.5,0) {$c_i^j$};
\node at (2.5,-5) {$-1$};
\node at (2.5,-1) {$-t$};
\node at (2,-3.5) {$1$};
\node at (2,-1.5) {$1$};
\node at (3,-4.5) {$t$};
\node at (3,-2.5) {$t$};
\end{tikzpicture}
\caption{A generic region $R_i^j$}
\label{genrij}
}
\qquad \qquad \qquad
\begin{minipage}{5.5cm}
\centering
\begin{tikzpicture}[scale = .5]
\braid[strands=4,braid start={(0,0)}]
{\sigma_2^{-1} \sigma_1^{-1} \sigma_3^{-1} \sigma_1^{-1} \sigma_3^{-1} \sigma_2^{-1} \sigma_1^{-1} \sigma_2^{-1}}
\draw[thick] (6,0)--(6,-8);
\draw[thick] (7,0)--(7,-8);
\draw[thick] (8,0)--(8,-8);
\draw[thick] (9,0)--(9,-8);
\draw[thick] (6,0) arc (0:180:1);
\draw[thick] (7,0) arc (0:180:2);
\draw[thick] (8,0) arc (0:180:3);
\draw[thick] (9,0) arc (0:180:4);
\draw[thick] (6,-8) arc (180:0:-1);
\draw[thick] (7,-8) arc (180:0:-2);
\draw[thick] (8,-8) arc (180:0:-3);
\draw[thick] (9,-8) arc (180:0:-4);
\node at (6,-4) {$\vee$};
\node at (7,-4) {$\vee$};
\node at (8,-4) {$\vee$};
\node at (9,-4) {$\vee$};
\node at (4,-2.5) {$\bullet$};
\node at (3,-5.5) {$\bullet$};
\node at (2,-1.5) {$\bullet$};
\node at (3.5,-4) {$\circ$};
\node at (2.5,0) {$\circ$};
\node at (2.5,-7) {$\circ$};
\node at (1.5,-3) {$\circ$};
\node at (1.5,-6) {$\circ$};
\node at (0.5,-4) {$\ast$};
\node at (1.5,-8) {$\ast$};
\node at (1.5,0) {$\ast$};
\end{tikzpicture}
\caption{The Kauffman state with weight $\pm 1$}
\label{kauffmanpm1}
\end{minipage}
\end{figure}

Next we position the white circles in Figure \ref{kauffmanpm1}. Observe that the regions $R_1^j$ have no ``left" crossings, so they have no 1 labels. It follows that we must assign $R_1^j$ to crossing $c_1^{j-1}$, at which it has its unique -1 label. Now move to the regions $R_2^j$. We have used all of the $c_1^j$ crossings, so there are no more ``left" crossings available. So again, assign $R_2^j$ to $c_2^{j-1}$, at which it has its unique -1 label. Iterating this procedure yields a state with $(n-1)$ `1' labels, and $(\ell-n+1)$ `-1' labels. At each step, the choice is forced, so this is a unique state.
\end{proof}
\begin{lemma}
There is  a unique state with weight $\pm t^{\ell-n+1}$.
\label{statet}
\end{lemma}
\begin{proof}
We begin by noting that, for fixed $i$, we cannot consistently choose $\pm t$ at crossing $c_i^j$ for all $j$. Such a choice would use $R_i^k$ or $R_{i-1}^k$ regions. Then, to the right of these crossings, there would be only \[s_i:=n_{i+1} + n_{i+2}=\cdots n_{n-1}\] crossings, but $s_{i}+1$ adjacent regions to be matched. Since this cannot happen, it follows that for each $i$, at least one of the $c_i^j$ must have a $\pm 1$ label \footnote{Actually, a similar argument shows that at least one must have a $+1$ label, but we do not need this distinction.}. Therefore, the highest power of $t$ which can be assigned to a state is $\ell-n+1$. In order to achieve this, it is necessary that exactly one $c_i^j$ has a $\pm 1$ label for each $i$.\\

\begin{figure}[h!]
\centering
\begin{tikzpicture}[scale = .5]
\braid[strands=4,braid start={(0,0)}]
{\sigma_2^{-1} \sigma_1^{-1} \sigma_3^{-1} \sigma_1^{-1} \sigma_3^{-1} \sigma_2^{-1} \sigma_1^{-1} \sigma_2^{-1}}
\draw[thick] (6,0)--(6,-8);
\draw[thick] (7,0)--(7,-8);
\draw[thick] (8,0)--(8,-8);
\draw[thick] (9,0)--(9,-8);
\draw[thick] (6,0) arc (0:180:1);
\draw[thick] (7,0) arc (0:180:2);
\draw[thick] (8,0) arc (0:180:3);
\draw[thick] (9,0) arc (0:180:4);
\draw[thick] (6,-8) arc (180:0:-1);
\draw[thick] (7,-8) arc (180:0:-2);
\draw[thick] (8,-8) arc (180:0:-3);
\draw[thick] (9,-8) arc (180:0:-4);
\node at (6,-4) {$\vee$};
\node at (7,-4) {$\vee$};
\node at (8,-4) {$\vee$};
\node at (9,-4) {$\vee$};
\node at (4,-4.5) {$\bullet$};
\node at (3,-5.5) {$\bullet$};
\node at (2,-6.5) {$\bullet$};
\node at (3.5,-3) {$\circ$};
\node at (2.5,-1) {$\circ$};
\node at (2.5,-8) {$\circ$};
\node at (1.5,-2) {$\circ$};
\node at (1.5,-4) {$\circ$};
\node at (0.5,-4) {$\ast$};
\node at (1.5,-8) {$\ast$};
\node at (1.5,0) {$\ast$};
\end{tikzpicture}
\caption{The Kauffman state with weight $\pm t^{\ell-w+1}$}
\label{kauffmanlw1}
\end{figure}
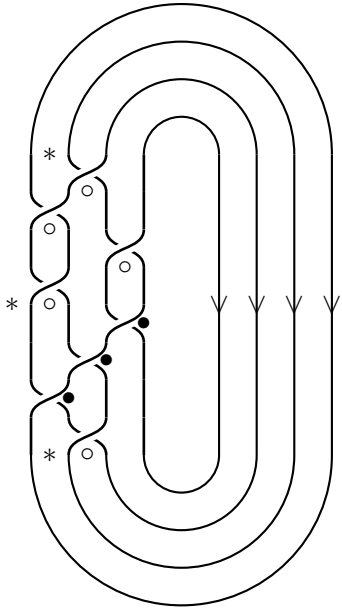

Consider first the crossings $c_1^j$. All of the $t$ labels at these crossings are in the starred region $R_0$. The only way to get the maximum $n_1-1$ labels of $\pm t$ is to assign $c_1^j$ to $R_1^j$ for all $j>1$. In this case, the only remaining choice at the crossing $c_1^1$, which we will denote by $c_1^{k_1}$, is the label $+1$, coming from some region, which we will call $R_2^{k_2}$. \\
Then consider the crossings $c_2^j$. Again, since all of the $R_1^j$ regions have already been assigned to crossings, there are no $t$ labels available. The best we can do is get $n_2-1$ labels of $-t$, by assigning crossing $c_2^j$ to region $R_2^j$. This can be done for all but one $j$, because one of the regions, $R_2^{k_2}$, was already assigned to crossing $c_1^{k_1}$. This leaves the crossing $c_2^{k_2}$ without a label, and the only remaining choice is the $+1$ label coming from a region which we will call $R_3^{k_3}$. \\
The same process uniquely determines the assignments for all other crossings. Eventually we get a crossing $c_{n-1}^{k_{n-1}}$, which can only be assigned to the region $R_n$. This state has $(n-1)$ labels of $+1$, where the crossing $c_i^{k_i}$ is assigned to region $R_{i+1}^{k_{i+1}}$, and $(\ell-n+1)$ labels of $-t$, where the remaining crossings $c_i^j$ are assigned to regions $R_i^j$. This is the unique state with weight $\pm t^{\ell-n+1}$, shown in Figure \ref{kauffmanlw1}.
\end{proof}
\begin{proof}[\sc Proof of Theorem \ref{main}.] Lemma \ref{upperbound} shows that $\ell-n+1$ is an upper bound for the depth. Lemmas \ref{state1} and \ref{statet} together show that \[Br(\Delta_L(t))=\ell-n+1,\] and so by inequality \eqref{breadthineq}, this quantity is also a lower bound for the depth.
\end{proof}
More generally, the closure of a positive -- but not strictly positive -- braid is a split link, splitting into closures of strictly positive braids, whose depths are determined by Theorem \ref{main}. Note also that the same argument applies to closures of negative braids.

\bibliography{braid}{}
\bibliographystyle{amsplain} 




\end{document}